\documentclass[11pt]{article} 

\usepackage[latin1,utf8]{inputenc}
\usepackage[T1]{fontenc} 
\usepackage{textcomp} 
\usepackage[english]{babel}
\usepackage[autolanguage]{numprint} 
\usepackage{hyperref} 
\usepackage{graphicx} 
\usepackage{subcaption}
\usepackage{epstopdf}
\usepackage{verbatim} 
\usepackage[all]{xy}
\usepackage{float}

\usepackage{lmodern}
\usepackage{mathrsfs}

\usepackage{amsmath,amssymb,amsthm} 
\usepackage{tikz,tkz-tab} 

\usepackage{mathdots}
\usepackage{makeidx}
\usepackage{stmaryrd} 

\usepackage{enumitem}
\setenumerate[0]{label=(\roman*)}

\renewcommand\epsilon\varespilon 
\renewcommand\phi\varphi 
\renewcommand\rho\varrho 
\renewcommand \subset \subseteq

\newcommand\balpha{\ba}

\newcommand\AND{\quad\textrm{and}\quad}

\newcommand\ba{\mathbf{a}}  
\newcommand\bb{\mathbf{b}}  

\newcommand\bQ{\mathbb{Q}} 
\newcommand\bR{\mathbb{R}} 
\newcommand\bZ{\mathbb{Z}} 

\newcommand\cA{\mathcal{A}}
\newcommand\cB{\mathcal{B}}
\newcommand\cC{\mathcal{C}}

\newcommand\cK{\mathcal{K}}

\newcommand\cM{\mathcal{M}}

\newcommand\CL{\mathrm{L}} 
\newcommand\ee{\varepsilon}

\newcommand\GrO{\mathcal{O}} 
\newcommand\homega{\widehat{\omega}} 

\newcommand\tP{\widetilde{P}}

\newcommand\minF{F_n}

\newcommand\minba{\ba}
\newcommand\minx{x}
\newcommand\miny{y}

\newcommand\mint{\theta}

\newcommand\ie{\textsl{i.e. }} 
\newcommand\eg{\textsl{e.g. }} 

\newcommand{\norm}[1]{\|#1 \|} 
\newcommand\PP{\mathcal{P}} 

\newcommand\vol{\mathrm{vol}}

\newcommand\la{\langle}
\newcommand\ra{\rangle}
\newcommand{\Vect}[2][]{\left\la #2\right\ra_{#1}} 

\newcommand\pu{\underline{\phi}} 
\newcommand\po{\overline{\phi}} 

\theoremstyle{definition} 
\newtheorem{Def}{Definition}[section]

\theoremstyle{plain} 
\newtheorem{Prop}[Def]{Proposition} 
\newtheorem{Lem}[Def]{Lemma} 
\newtheorem{Thm}{Theorem}[section] 
\newtheorem{Cor}[Def]{Corollary} 

\theoremstyle{remark} 
\newtheorem{Rem}[Def]{Remark} 

\numberwithin{equation}{section} 

\title{On approximation to a real number by algebraic numbers of bounded degree}  
\author{Anthony Poëls} 
\date{}

\newcommand{\Addresses}{{
  \bigskip
  \footnotesize
  \noindent\textsc{Universite Claude Bernard Lyon 1, \\
  Institut Camille Jordan UMR 5208, \\
  69622 Villeurbanne,\\
  France}
  \medskip
  \par\nopagebreak
  \noindent\textit{E-mail address}: \texttt{poels@math.univ-lyon1.fr}
}}

\newcommand{\MSC}{{
  \footnotesize
  \textbf{MSC~2020}: 11J13(Primary), 11J82 (Secondary).
}}

\newcommand{\keysW}{{
  \footnotesize
  \textbf{Keywords}: Wirsing conjecture, approximation by algebraic numbers, polynomial approximation,
  exponents of approximation, transcendance.
}}

\newcommand{\Ack}{{
  \footnotesize

  \textbf{Acknowledgements}: I am very grateful to Damien Roy for his encouragement and the many fruitful
  discussions we had. His careful reading of this work and his many pertinent suggestions helped
  improve the overall presentation and simplify some arguments.
}}

\begin{document} 

\baselineskip=17pt 

\maketitle

\begin{abstract}
    In his seminal 1961 paper, Wirsing studied how well a given transcendental real number $\xi$ can be approximated by algebraic numbers $\alpha$ of degree at most $n$ for a given positive integer $n$, in terms of the so-called naive height $H(\alpha)$ of $\alpha$. He showed that the infimum $\omega^*_n(\xi)$ of all $\omega$ for which infinitely many such $\alpha$ have $|\xi-\alpha| \le H(\alpha)^{-\omega-1}$ is at least $(n+1)/2$. He also asked if we could even have $\omega^*_n(\xi) \ge n$ as it is generally expected. Since then, all improvements on Wirsing's lower bound were of the form $n/2+\GrO(1)$ until Badziahin and Schleischitz showed in 2021 that $\omega^*_n(\xi) \ge an$ for each $n\ge 4$, with $a=1/\sqrt{3}\simeq 0.577$. In this paper, we use a different approach partly inspired by parametric geometry of numbers and show that $\omega^*_n(\xi) \ge an$ for each $n\ge 2$, with $a=1/(2-\log 2)\simeq 0.765$.
\end{abstract}

\MSC

\keysW

\section{Introduction}

One of the fundamental questions in Diophantine approximation is the following. Given an irrational real number $\xi$, how well can it be approximated by rational numbers? A simple application of Dirichlet's box principle ensures that there exist infinitely many rational numbers $p/q$ with $q\geq 1$ and
\begin{align}
    \label{eq : Dirichlet dim one}
    \left| \xi - \frac{p}{q} \right| \leq \frac{1}{q^2}.
\end{align}
A result which goes back to Khintchine \cite{khintchine1926klasse} ensures that the above property is optimal in the following sense. For any fixed $\ee > 0$, the set of real number $\xi$ for which there exist infinitely many $p/q$ with $|\xi-p/q| \leq 1/q^{2+\ee}$ has Lebesgue measure zero. If we think of rational numbers as algebraic numbers of degree one, then it is natural to generalize the previous question in the following way: given a positive integer $n$, how well can $\xi$ be approximated by algebraic numbers of degree at most $n$? In \cite{koksma1939mahlersche} Koksma introduced a classification of real numbers in terms of the behaviour of
the sequence $(\omega_n^*(\xi))_{n\geq 1}$, where the classical exponent $\omega_n^*(\xi)$ is defined as the supremum of the real numbers $\omega^*>0$ for which the inequalities
\begin{align}
    \label{eq : def omega_n^*}
    0 < |\xi-\alpha| \leq H(\alpha)^{-\omega^*-1}
\end{align}
admit infinitely many solutions in algebraic numbers $\alpha$ of degree at most $n$. Here, $H(\alpha)$ denote the (naive) \textsl{height} of $\alpha$, that is the largest absolute value of the coefficients of its irreducible polynomial over $\bZ$. See \cite[Chapter~VIII, Section 9]{schmidt1996diophantine} and \cite[Section~2]{bugeaud2015exponents} for a motivation of the summand $-1$ appearing in the exponent in \eqref{eq : def omega_n^*}. By a result of Sprind\v zuk \cite{sprindzuk1969mahler} combined with classical transference inequalities (see \cite[Chapter~VIII, Section 9]{schmidt1996diophantine} and \cite[Theorem~2.8]{bugeaud2015exponents}), we have
\begin{align}
    \label{eq : valeur generic de omega_n^*}
    \omega_n^*(\xi) = n
\end{align}
for almost all real numbers $\xi$ with respect to Lebesgue measure. Schmidt's Subspace theorem implies that \eqref{eq : valeur generic de omega_n^*} also holds if $\xi$ is algebraic of degree $\geq n+1$ (see \cite[Chapter~6, Corollary~1E]{schmidt1996diophantine}). However, given a specific transcendental real number $\xi$, it is usually extremely difficult to determine $\omega_n^* (\xi)$. We can find in Wirsing's original 1961 paper \cite{wirsing1961approximation} the following famous problem, which is the main motivation for the present work.

\medskip

\noindent\textbf{Wirsing's problem.} Do we have $\omega_n^*(\xi)\geq n$ for any integer $n\geq 1$ and any transcendental real number $\xi$?

\medskip

So far, and despite a lot of effort, it has been confirmed only for $n = 1$ (this is a consequence of \eqref{eq : Dirichlet dim one}) and for $n = 2$ (by Davenport and Schmidt \cite{davenport1967approximation}, also see \cite{davenport1968theorem}). In his 1961 paper, Wirsing also established the following lower bound
\begin{align}
    \label{eq intro: borne Wirsing}
    \omega_n^*(\xi) \geq \frac{n+1}{2},
\end{align}
valid for each transcendental real number $\xi$. Until very recently, the best lower bounds due to Bernik and Tishchenko \cite{bernik1993integral} and \cite{tishchenko2000approximation, tsishchanka2007approximation} were of the form $n/2+\GrO(1)$. In 2021, Badziahin and Schleischitz made a important breakthrough \cite{badziahin2021improved} by improving on the factor $1/2$ for the first time. For precisely, they showed that for each $n\geq 4$ and each transcendental real number $\xi$, we have
\begin{align*}
    \omega_n^*(\xi) \geq an, \qquad \textrm{where } a = \frac{1}{\sqrt 3} = 0.577\cdots
\end{align*}
Our main result improves the above result as follows.

\begin{Thm}
    \label{Thm : main}
    Let $n$ be an integer $\geq 2$. For any transcendental real number $\xi$, we have
    \begin{align*}
        \omega_n^*(\xi) \geq an,\qquad \textrm{where } a = \frac{1}{2-\log 2} = 0.765\cdots.
    \end{align*}
\end{Thm}

Note that our bounds are better than those obtained in \cite{tsishchanka2007approximation} starting with $n=7$. We believe that the constant $a$ in Theorem~\ref{Thm : main} is not optimal and could be improved by refining our method.

\medskip

Given a transcendental real number $\xi$, Wirsing's approach to showing his lower bound \eqref{eq intro: borne Wirsing} is to construct coprime polynomials $P$ and $Q$ of degree at most $n$, which have integer coefficients and have very small absolute values at $\xi$. Considering their resultant, he then proves that a root of $P$ or $Q$ must be very close to $\xi$. For the proof of Theorem~\ref{Thm : main}, the key-point is to consider simultaneously $n+1$~linearly independent polynomials $P_1,\dots,P_{n+1} \in \bZ[X]_{\leq n}$, instead of just two. This idea has its origins in \cite{poels2024pol}, where we improve the upper bound for the uniform exponent of polynomial approximation.

\medskip

This paper is organized as follows. In Sections~\ref{section : parametric geometry of numbers} and~\ref{section : construction points à considérer}, we construct the aforementioned polynomials $P_i$, which roughly realize the successive minima of a certain symmetric convex body in $\bR[X]_{\leq n}$ (with respect to the lattice of integer polynomials $\bZ[X]_{\leq n}$). We are able to control rather precisely their size and their absolute value at $\xi$. In Section~\ref{Section : semi-resultants}, by evaluating some kind of non-zero generalized resultant, we prove that for each $k=2,\dots,n+1$, at least one of the polynomials $P_1,\dots,P_k$ has a root very close to $\xi$. Taking into account all these approximations, we then conclude in Section~\ref{section: introduction de la fonction F} that $\omega_n^*(\xi)$ is bounded below by the minimum of an explicit function of $n+1$~variables. In the last two Sections~\ref{Section: points realisant le min} and~\ref{section: minration finale}, which are independent from the previous ones, we deal with the optimization problem of finding this minimum. We show that it is at least equal to $n/(2-\log(2))$.

\section{Notation}
\label{Section : notation}

Given a ring $A$ (typically $A=\bR$ or $\bZ$) and an integer $n\geq 0$, we denote by $A[X]$ the ring of polynomials in $X$ with coefficients in $A$, and by $A[X]_{\leq n} \subset A[X]$ the subgroup of polynomials of degree at most $n$. We say that $P\in\bZ[X]$ is \textsl{primitive} if it is non-zero and the greatest common divisor of its coefficients is $1$. Given $P(X)=\sum_{k=0}^{n} a_kX^k \in \bR[X]$, we set
\begin{align*}
    \norm{P} = \max_{0\leq k \leq n} |a_k|.
\end{align*}
For $k=0,\dots,n$, we define
\begin{align*}
    P^{[k]} = \frac{1}{k!} \frac{d^{k\,} P}{d X^k} \in \bR[X].
\end{align*}
Then, for each real number $\xi$, we have
\begin{align*}
    P(X) = \sum_{k=0}^{n} P^{[k]}(\xi) (X-\xi)^k.
\end{align*}
For short, we say that polynomials of $\bR[X]_{\leq n}$ or $\bZ[X]_{\leq n}$ are \textsl{linearly independent} to mean that they are linearly independent over $\bR$. We identify $\bR^{n+1}$ to $\bR[X]_{\leq n}$ via the isomorphism
\[
    (a_0,\dots,a_n)\longmapsto a_0+a_1X+\cdots +a_nX^n.
\]
Then, the volume $\vol(C)$ of a closed set $C\subset\bR[X]_{\leq n}$ is simply the Lebesgue measure of the corresponding set in $\bR^{n+1}$.

\medskip

Let $\xi\in\bR$ be a transcendental number and $n$ a positive integer. The following two classical Diophantine exponents will play an important role in our study. We denote by $\homega_n(\xi)$ (resp. $\omega_n(\xi)$), the supremum of the real numbers $\omega > 0$ such that the system
\begin{align*}
    \norm{P} \leq H \AND 0 <|P(\xi)| \leq H^{-\omega}
\end{align*}
admits a non-zero solution $P\in\bZ[X]_{\leq n}$ for each large enough $H$ (resp. for arbitrarily large $H$). Dirichlet's Theorem implies that
\begin{align*}
    n \leq \homega_n(\xi) \leq \omega_n(\xi)
\end{align*}
(see for example \cite[Chapter 2, Theorem 1C]{schmidt1996diophantine}). The exponent $\omega_n^*(\xi)$, defined as in the introduction, is the supremum of the real numbers $\omega^*>0$ for which there are infinitely many algebraic numbers $\alpha$ of degree at most $n$ satisfying
\begin{align*}
    0 < |\xi-\alpha| \leq H(\alpha)^{-\omega^*-1}.
\end{align*}
Here, $H(\alpha) = \norm{P_\alpha}$, where $P_\alpha$ is the minimal polynomial of $\alpha$ irreducible over $\bZ$ (with positive leading coefficient). The reader may consult \cite{bugeaud2007exponents} for an interesting survey presenting, among others, several transference inequalities between these exponents.

\begin{Rem}
    \label{Rem: section notation, exposants non extremal}
    According to \cite[Theorems~2.6 and~3.1]{bugeaud2015exponents}, we have
    \begin{align*}
        \omega_n^*(\xi) \geq \omega_n(\xi)-n+1,
    \end{align*}
    so that if $\omega_n(\xi) = \infty$, then $\omega_n^*(\xi)  = \infty \geq n$. Consequently, in the following we will often assume that $\omega_n(\xi) < \infty$.
\end{Rem}

By Gelfond's Lemma (see \eg \cite[Lemma A.3]{bugeaud2004approximation} as well as \cite{brownawell1974sequences}), for each non-zero $P,Q\in\bZ[X]_{\leq n}$, if $P$ divides $Q$, then
\begin{equation}
    \label{eq: Gelfond P facteur de Q}
    e^{-n} \norm{P} < \norm{Q}.
\end{equation}
In particular, if $\norm{Q} \leq e^{-n} \norm{P}$, then $P$ cannot be a factor of $Q$.

\medskip

Finally, given two functions $f, g : I \rightarrow [0, +\infty)$ on a set $I$, we write $f = \GrO(g)$ or $f \ll g$ or $g \gg f$ to mean that there is a positive constant $c$ such that $f (x) \leq cg(x)$ for each $x\in I$. We write $f \asymp g$ when both $f \ll g$ and $g \ll f$ hold.

\section{Parametric geometry of numbers}
\label{section : parametric geometry of numbers}

Let $\xi$ be a transcendental real number and $n$ be an integer $\geq 2$. Schmidt and Summerer's parametric geometry of numbers \cite{Schmidt2009, Schmidt2013}, \cite{Roy_juin} is a powerful tool for studying Diophantine exponents. Although we do not need much of this theory, it provides a convenient framework to state the results we need. In this section we first recall some elementary results from parametric geometry of numbers, then we establish several lemmas which form the basis of our future polynomial constructions.

\medskip

Following the approach of Roy \cite{Roy_juin} (with the maximum norm instead of the Euclidean norm), we consider for any parameter $q\geq 0$ the symmetric convex body
\begin{align*}
    \cC_\xi(q) = \Big\{P\in\bR[X]_{\leq n} \,;\, \norm{P} \leq 1 \AND |P(\xi)|\leq e^{-q} \Big\}.
\end{align*}
For $i=1,\dots,n+1$, we define $L_i(q)$ as the smallest real number $L$ such that $e^{L}\cC_\xi(q)\cap \bZ[X]_{\leq n}$ contains at least $i$ linearly independent polynomials. Thus, $e^{L_1(q)},\dots,e^{L_{n+1}(q)}$ are the successive minima of $\cC_\xi(q)$ with respect to the lattice $\bZ[X]_{\leq n}$. We group these minima in a map $\CL_\xi:[0,\infty)\rightarrow \bR^{n+1}$ defined by
\begin{align*}
    \CL_\xi(q) = \big(L_1(q),\cdots,L_{n+1}(q)\big).
\end{align*}
Recall that the functions $L_i$ are continuous, piecewise linear with slopes $0$ and $1$ (they are therefore non-decreasing). Furthermore, since $\vol\big( \cC_\xi(q) \big)\asymp e^{-q}$, Minkowski's second theorem implies that
\begin{align*}
    L_1(q) + \cdots + L_{n+1}(q) = q + \GrO(1), \qquad q\in[0,\infty),
\end{align*}
where the implicit constant depends on $n$ and $\xi$ only. To any non-zero polynomial $P\in\bZ[X]_{\leq n}$ we associate a function $L(P,\cdot)\rightarrow [0,+\infty)$ by setting
\begin{align*}
    L(P,q) = \max\big\{\log \norm{P}, q+\log |P(\xi)| \big\} \qquad (q\in[0,+\infty)).
\end{align*}
Following Roy's terminology \cite[\S 2.2]{Roy_juin}, the \textsl{trajectory} of a non-zero polynomial $P\in\bZ[X]_{\leq n}$ is the graph of the function $L(P,\cdot)$. Note that $L(P,\cdot)$ is continuous, piecewise linear, constant on $[0,q_P]$ and increasing with slope $1$ on $[q_P,\infty)$, where the slope change point $q_P$ is
\begin{align*}
    q_P = \log \norm{P}-\log |P(\xi)|.
\end{align*}
Thus, for each $q\geq 0$, we have
\begin{align*}
    L(P,q) = \left\{\begin{array}{cc}
        \log \norm{P} & \textrm{if $q\leq q_P$}, \\
        \\
        q+\log |P(\xi)| & \textrm{if $q\geq q_P$}.
    \end{array} \right.
\end{align*}
Since, for each $q\geq 0$, the smallest $L\geq 0$ such that $P\in e^L\cC(q)$ is precisely $L(P,q)$, we have
\begin{align}
    \label{eq: def L_1 as min...}
    L_1(q) = \min_{P\in \bZ[X]_{\leq n}\setminus \{0\}} L(P,q).
\end{align}
Moreover, since $\xi$ is transcendental, we have $\lim_{q\rightarrow\infty} L_1(q) = \infty$. Although we will not need them, we have the classical formulas (arguing as in~\cite[Theorem 1.4]{Schmidt2009})
\begin{align*}
    \pu= \liminf_{q\rightarrow \infty} \frac{L_1(q)}{q} = \frac{1}{1+\omega_n(\xi)} \AND \po=\limsup_{q\rightarrow \infty} \frac{L_1(q)}{q} = \frac{1}{1+\homega_n(\xi)}.
\end{align*}
The exponents $\pu$ and $\po$ are parametric versions of the exponents $\omega_n(\xi)$ and $\homega_n(\xi)$.

\begin{Lem}
    \label{Lem: estimation Q(xi) qui réalise L_1}
    Fix $\homega < \homega_n(\xi)$. There exists $q_0 = q_0(\homega) \geq 0$ with the following property. Let $q\in[q_0,\infty)$ and $Q\in\bZ[X]_{\leq n}$ be such that $L(Q,\cdot)$ has slope $1$ on $[q,\infty)$ and coincides with $L_1$ at $q$. Then
    \begin{align*}
        |Q(\xi)| \leq e^{-\homega L_1(q)}.
    \end{align*}
\end{Lem}

\begin{proof}
    Choose $q\geq 0$ and $Q\in\bZ[X]_{\leq n}$ such that $L(Q,\cdot)$ has slope $1$ on $[q,\infty)$. This means that $L(Q,q) =  q+\log |Q(\xi)|$. We also assume that $L(Q,q) = L_1(q)$ and set $H = e^{L_1(q)}$. By definition of $\homega_n(\xi)$, if $q$ is large enough, there exists a non-zero $P\in\bZ[X]_{\leq n}$ such that
    \begin{align*}
        \norm{P} < H = e^{L_1(q)} \AND |P(\xi)| \leq H^{-\homega}.
    \end{align*}
    Since $L_1(q)\leq L(P,q) = \max\{\log \norm{P}, q+\log |P(\xi)| \}$, this yields $L(P,q) = q+\log |P(\xi)|$, and
    \begin{align*}
        q+\log |Q(\xi)| = L_1(q) \leq q+\log|P(\xi)| \leq q-\homega L_1(q).
    \end{align*}
\end{proof}

\begin{Lem}
    \label{Lem: estimation det facile}
    There exists a constant $c>0$ which depends on $n$ and $\xi$ only such that, for any linearly independent polynomials $P_1,\dots,P_{n+1}\in\bZ[X]_{\leq n}$, we have
    \begin{align*}
        1 \leq  c\norm{P_1}\cdots\norm{P_{n+1}}\sum_{i=1}^{n+1}\frac{|P_i(\xi)|}{\norm{P_i}}.
    \end{align*}
\end{Lem}

\begin{proof}
    Since $\det(P_1,\dots,P_{n+1})$ is a non-zero integer, we have
    \begin{align*}
        1 \leq \left|\det(P_1,\dots,P_{n+1})\right| = \left|\det\left( P^{[i-1]}_j(0) \right)_{1\leq i,j\leq n+1}\right|
        = \left|\det\left( P^{[i-1]}_j(\xi) \right)_{1\leq i,j\leq n+1}\right|.
    \end{align*}
    We conclude by expanding the last determinant and by noting that for $j=1,\dots,n+1$, we have $P^{[0]}_j(\xi) = P_j(\xi)$ and
    $|P^{[i-1]}_j(\xi)| \ll \norm{P_j}$ ($i=2,\dots,n+1$).
\end{proof}

The following result is crucial for our approach. Under some condition, it provides $n+1$ linearly independent polynomials with integer coefficients which have ``good'' properties: their absolute values are small at $\xi$ and their height are under control. In some way, it is reminiscent of \cite[Theorem~3.1]{Roy_juin}. The idea is to start with a family of polynomials which realize the successive minima of $\cC_\xi(q)$, and then to correct these polynomials to make their absolute values small at~$\xi$.

\begin{Lem}
    \label{Lem base construction}
    Let $q\in[0,\infty)$ and $Q\in \bZ[X]_{\leq n}$ such that $L_1(q) = L(Q,q)$. We suppose that $L(Q,\cdot)$ has slope $1$ on $[q,+\infty)$. Then, there exist linearly independent polynomials $P_1,\dots,P_{n+1}\in\bZ[X]_{\leq n}$ such that $P_1=Q$ and
    \begin{enumerate}
      \item \label{item: lem base construction 1} $|P_i(\xi)| < |P_1(\xi)|$ and $e^{L_i(q)} \leq \norm{P_i} \leq 2e^{L_i(q)}$ for $i=2,\dots,n+1$;
      \item \label{item: lem base construction 2} $\norm{P_1}\leq   \cdots \leq \norm{P_{n+1}}$;
      \item \label{item: lem base construction 3} $|P_1(\xi)|\cdot\norm{P_2}\cdots \norm{P_{n+1}} \asymp 1$, with implicit constants depending only on $n$ and $\xi$.
    \end{enumerate}
\end{Lem}

\begin{proof}
    Let $Q_1=Q,Q_2,\dots,Q_{n+1} \in\bZ[X]_{\leq n}$ be linearly independent polynomials which realize $L_1(q),\cdots,L_{n+1}(q)$, \ie such that
    \[
        L(Q_i,q) = L_i(q) \qquad(i=1,\dots,n+1).
    \]
    By hypothesis on $Q=Q_1$, we have $q\geq q_1$, where $q_1 = \log \norm{Q_1} - \log |Q_1(\xi)|$ is the abscissa where $L(Q_1,\cdot)$ changes slope. We obtain
    \begin{align*}
        L_1(q) = L(Q_1,q) = \log \norm{Q_1} + q - q_1 \AND \log |Q_1(\xi)| = L_1(q) -q.
    \end{align*}
    Then, Minkowski's second theorem yields
    \begin{align}
        \label{eq proof: lemme 3.3 Minkowski}
        |Q_1(\xi)|e^{L_2(q)+\cdots + L_{n+1}(q)} = e^{-q+L_1(q)+\cdots + L_{n+1}(q)} \ll 1.
    \end{align}
    Set $P_1 = Q_1$, and for $i=2,\dots,n+1$ set
    \begin{align*}
        R_i = Q_i - \left\lfloor \frac{Q_i(\xi)}{P_1(\xi)} \right\rfloor P_1 \in \bZ[X]_{\leq n}.
    \end{align*}
    We have $|R_i(\xi)| < |P_1(\xi)|$, as well as
    \begin{align*}
        \norm{R_i} \leq 2 \max\left\{\norm{Q_i}, \frac{|Q_i(\xi)|}{|P_1(\xi)|} \cdot \norm{P_1} \right\} = 2e^{L(Q_i,q_1)} \leq 2e^{L(Q_i,q)} =  2e^{L_i(q)}.
    \end{align*}
    Denote by $P_2,\dots,P_{n+1}$ the polynomials $R_2,\dots,R_{n+1}$ reordered by increasing norm. By the above, for $i=2,\dots,n+1$, we have
    \begin{align}
        \label{eq proof: lemme 3.3 eq 2}
        |P_i(\xi)| < |P_1(\xi)| \AND \norm{P_i} \leq 2e^{L_i(q)}.
    \end{align}
    On the other hand, since $\log |P_i(\xi)|+q < \log |P_1(\xi)| +q = L_1(q) \leq L(P_i,q)$ (the last inequality coming from the minimality property of~\eqref{eq: def L_1 as min...}), we must have $L(P_i,q) = \log \norm{P_i}$, thus
    \begin{align*}
        \log \norm{P_1} \leq L(P_1,q) = L_1(q) \leq L(P_i,q) =\log \norm{P_i}.
    \end{align*}
    So, we have
    \begin{align*}
        L(P_1,q) \leq L(P_2,q) = \log \norm{P_2} \leq \cdots \leq \log \norm{P_{n+1}} = L(P_{n+1},q).
    \end{align*}
    Since the polynomials $P_1,P_2,\dots,P_{n+1} \in\bZ[X]_{\leq n}$ are linearly independent, we deduce that
    \begin{align}
        \label{eq proof: lemme 3.3 eq 3}
        L_i(q) \leq L(P_i,q) = \log \norm{P_i} \qquad \textrm{for } i=2,\dots,n+1.
    \end{align}
    So the conditions~\ref{item: lem base construction 1} and~\ref{item: lem base construction 2} are fulfilled. Finally, Lemma~\ref{Lem: estimation det facile} together with \eqref{eq proof: lemme 3.3 eq 2} and~\eqref{eq proof: lemme 3.3 Minkowski} yields
    \begin{align*}
        1 \ll \prod_{i=1}^{n+1}\norm{P_i}\sum_{i=1}^{n+1}\frac{|P_i(\xi)|}{\norm{P_i}} \ll |P_1(\xi)|\prod_{i=2}^{n+1}\norm{P_i} \ll |P_1(\xi)|e^{L_2(q)+\dots+L_{n+1}(q)} \ll 1.
    \end{align*}
    Note that \eqref{eq proof: lemme 3.3 eq 3} together with \eqref{eq proof: lemme 3.3 eq 2} show that $L_i(q) \leq L(P_i,q) \leq L_i(q)+\log 2$ for $i=1,\dots,n+1$, while $L_1(q) = L(P_1,q)$. Thus, roughly speaking, the polynomials $P_i$ realize the successive minima of $\cC_\xi(q)$ up to a factor $\leq 2$.

\end{proof}

\section{Families of polynomials}
\label{section : construction points à considérer}

Let $\xi$ be a transcendental real number and $n$ be an integer $\geq 2$. In this section, we suppose that
\begin{align*}
    \omega_n(\xi) < \infty
\end{align*}
(see Remark~\ref{Rem: section notation, exposants non extremal}). In this section, we start to relate our polynomial constructions to the exponents $\omega_n(\xi)$ and $\homega_n(\xi)$. Fix a small $\ee \in (0,1)$ and set
\begin{align}
    \label{eq def : homega et omega fct de epsilon}
    \homega = \homega(\ee) = \homega_n(\xi)-\frac{\ee}{2} \AND \omega = \omega(\ee) = \omega_n(\xi)- \frac{\ee}{2}.
\end{align}
It follows from the definition of $\omega_n(\xi)$ and $\homega_n(\xi)$ that there exists $H_0\geq 1$ such that for each $H > H_0$, the system
\begin{equation}
    \label{eq : système pol}
    \norm{Q} \leq H \AND |Q(\xi)| \leq H^{-\homega}
\end{equation}
has a non-zero solution $Q\in\bZ[X]_{\leq n}$, and that any such $Q$ satisfies
\begin{align}
    \label{eq: minoration P(xi) sol}
    |Q(\xi)|\geq \norm{Q}^{-\omega_n(\xi)-\ee/2}
\end{align}
(because when $H$ goes to infinity, the quantity $|Q(\xi)|$ tends to $0$, and thus $\norm{Q}$ also goes to infinity). Define
\[
    \PP(\ee) = \left\{P\in\bZ[X]_{\leq n} \textrm{ irreducible } \,;\, e^{-n}\norm{P}\geq H_0 \AND |P(\xi)| \leq \norm{P}^{-\omega} \right\}.
\]
Note that any element of $\PP(\ee)$ has norm at least $e^nH_0 > 1$. A classical argument of Wirsing \cite[Hilfssatz 4]{wirsing1961approximation} ensures that the set $\PP(\ee)$ is infinite (see also \cite[Section 6]{davenport1969approximation}).
Now, write $\PP(\ee)$ as a disjoint union
\begin{align*}
  \PP(\ee) = \PP_0(\ee) \bigsqcup \PP_1(\ee),
\end{align*}
where
\begin{align*}
    \PP_0(\ee) = \left\{ P\in \PP(\ee) \,;\, \log\big(e^{-n} \norm{P} \big) < L_1(q_P) \right\} \AND \PP_1(\ee) = \PP(\ee)\setminus \PP_0(\ee),
\end{align*}
and $q_P = \log \norm{P} - \log |P(\xi)|$ as in Section~\ref{section : parametric geometry of numbers}. The set $\PP_0(\ee)$ is the set of polynomials $P\in\PP(\ee)$ which almost realize $L_1$ at $q=q_P$, since $L(P,q_P) < L_1(q_P)+n$. There is however no guarantee that this set is infinite.

\begin{Lem}
    \label{lem: estimations de base pour Wirsing}
    Let $\ee \in (0,1)$ and let $P\in \PP(\ee)$. There are linearly independent polynomials $Q_1,\dots,Q_{n+1}\in\bZ[X]_{\leq n}$     satisfying the following properties. Write $H_i = \norm{Q_i}$ for $i=1,\dots,n+1$.
    \begin{enumerate}
        \item The polynomials $Q_1$ and $Q_2$ are coprime and $Q_2 = P$.
        \item We have $H_1 \leq  \dots \leq H_{n+1}$, and there exists $x\geq n$ such that $H_2\cdots H_{n+1} = H_2^x$.
        \item \label{item:lem: estimations de base pour Wirsing:item 3} If $P\in \PP_0(\ee)$ and $\norm{P}$ is large enough, then $x \in [\homega_n(\xi)-\ee,\omega_n(\xi)+\ee]$ and
            \begin{align}
                \label{eq lem estimations de base pour Wirsing}
                \max\left\{|Q_1(\xi)|,\dots,|Q_{n+1}(\xi)|\right\} \ll  H_2^{-x}.
            \end{align}
        \item \label{item:lem: estimations de base pour Wirsing:item 4} If $P\in \PP_1(\ee)$ and $\norm{P}$ is large enough, then $x \in [\omega_n(\xi)-\ee,\omega_n(\xi)+\ee]$ and
            \begin{align}
                \label{eq lem estimations de base pour Wirsing bis}
                \max\left\{|Q_1(\xi)|,\dots,|Q_{n+1}(\xi)|\right\} \ll  H_2^{-x+\ee}.
            \end{align}
    \end{enumerate}
    The implicit constants depend on $n$ and $\xi$ only.
\end{Lem}

\begin{proof}
    Recall from \eqref{eq def : homega et omega fct de epsilon} that $\homega=\homega(\ee)$ and $\omega=\omega(\ee)$. Let $P\in \PP(\ee)$ and let $q \geq 0$ be maximal such that $L_1(q) = \log(e^{-n} \norm{P})$. The point $q$ tends to infinity as $\norm{P}$ goes to infinity. Let $Q\in \bZ[X]_{\leq n}$ be such that
    \begin{align*}
        L(Q,q) = L_1(q).
    \end{align*}
    By maximality of $q$, there exists $\eta > 0$ such that $L_1$ has slope $1$ on $[q,q+\eta]$. Since $L_1 \leq L(Q,\cdot)$, the function $L(Q,\cdot)$ has slope $1$ on $[q,\infty)$. Therefore
    \begin{align}
        \label{eq proof: lem construction cambouis 1}
        \log\norm{Q} \leq L(Q,q) = \log |Q(\xi)| + q = L_1(q) = \log\big(e^{-n}\norm{P}\big).
    \end{align}
    Thus $\norm{Q} \leq e^{-n} \norm{P}$, and, by \eqref{eq: Gelfond P facteur de Q}, the irreducible polynomial $P$ cannot be a factor of $Q$. They are therefore coprime. Moreover, Lemma~\ref{Lem: estimation Q(xi) qui réalise L_1} implies that if $q$ (or equivalently $\norm{P}$) is large enough, then $Q$ is solution of~\eqref{eq : système pol}, namely
    \begin{equation*}
        \norm{Q} \leq H \AND |Q(\xi)| \leq H^{-\homega},
    \end{equation*}
    with $H=e^{-n}\norm{P} \geq H_0$. Combined with \eqref{eq: minoration P(xi) sol}, this gives
    \begin{align}
        \label{eq proof: lem construction cambouis 3}
        \norm{Q}^{-\omega_n(\xi)-\ee/2} \leq |Q(\xi)| \leq \big(e^{-n}\norm{P}\big)^{-\homega} \ll \norm{P}^{-\homega}.
    \end{align}
    On the other hand, there exist $P_1,\dots, P_{n+1}$ in $\bZ[X]_{\leq n}$, with $P_1=Q$ satisfying assertions~\ref{item: lem base construction 1}--\ref{item: lem base construction 3} of Lemma~\ref{Lem base construction}. In particular, for $i=2,\dots,n+1$, we have
    \[
        \norm{P_i} \geq e^{L_1(q)} = e^{-n}\norm{P} \AND |P_i(\xi)| < |Q(\xi)|.
    \]
    For these indices $i$, set $\tP_i = P_i+\lambda_iP$ with $\lambda_i = 0$ if $\norm{P_i} > \norm{P}$, and $\lambda_i = 3$ otherwise, so that
    \[
        \norm{P_i} \asymp \norm{\tP_i} > \norm{P} \AND |\tP_i(\xi)| \leq 4\max\left\{ |Q(\xi)| , |P(\xi)|\right\},
    \]
    as well as
    \begin{align}
        \label{eq proof: lem construction cambouis 2}
        |P_1(\xi)| \prod_{i=2}^{n+1} \norm{\tP_i} \asymp |P_1(\xi)| \prod_{i=2}^{n+1} \norm{P_i} \asymp 1.
    \end{align}
    Since the family $P,P_1,\tP_2,\dots,\tP_{n+1}$ spans $\bR[X]_{\leq n}$ and since $P$ and $Q=P_1$ are linearly independent, there exists an index $j\in\{2,\dots,n+1\}$ such that $P,P_1,\tP_2,\dots,\widehat{\tP_j},\dots,\tP_{n+1}$ are linearly independent (where $\tP_j$ is omitted from the list). We denote by $Q_1,\dots,Q_{n+1}$ this family reordered by increasing norm. By construction of the polynomials $\tP_k$, we have $(Q_1,Q_2) = (Q,P)$. Let $x\in\bR$ be such that
    \[
        \norm{Q_2}\cdots \norm{Q_{n+1}} = \norm{Q_2}^x.
    \]
    The inequalities $\norm{Q_2} \leq \cdots \leq \norm{Q_{n+1}}$ imply that $x\geq n$. The two first assertions of the lemma are thus satisfied. We also have
    \begin{align}
        \label{eq proof: lem construction cambouis 2 bis}
        \max\left\{|Q_1(\xi)|,\dots,|Q_{n+1}(\xi)|\right\} \ll \max\{|Q(\xi)|, |P(\xi)|\}.
    \end{align}
    Since $\norm{P} \leq \norm{\tP_j}$, we deduce from \eqref{eq proof: lem construction cambouis 2} that
    \begin{align*}
        1 \asymp |P_1(\xi)| \prod_{i=2 }^{n+1} \norm{\tP_i} \geq |Q(\xi)|\prod_{i=2 }^{n+1} \norm{Q_i} = |Q(\xi)| \cdot \norm{P}^{x}.
    \end{align*}
    Therefore
    \begin{align}
        \label{eq proof: lem construction cambouis 5}
        |Q(\xi)| \ll \norm{P}^{-x}.
    \end{align}
    According to the first inequality of \eqref{eq proof: lem construction cambouis 3} (and since $\norm{Q} < \norm{P}$), we have
    \begin{align}
        \label{eq proof: lem construction cambouis 2 ter}
        \norm{P}^x \ll \norm{P}^{\omega_n(\xi)+\ee/2}.
    \end{align}
    Consequently, as soon as $\norm{P}$ is large enough, we have $x\in [n, \omega_n(\xi)+\ee]$. It remains to prove the last two assertions of our Lemma.

    \medskip

    \noindent\textbf{Proof of assertion \ref{item:lem: estimations de base pour Wirsing:item 3}.}
    Assume that $P\in\PP_0(\ee)$. If $P\in\PP_0(\ee)$, then $L_1(q_P) >  \log(e^{-n} \norm{P})$ and we find
     \[
        q < q_P = \log H(P)-\log |P(\xi)|.
    \]
    It follows that
    \begin{align*}
        \log\big(e^{-n} \norm{P}\big) = L_1(q) = \log |Q(\xi)| + q \leq \log |Q(\xi)| +  \log \norm{P}-\log |P(\xi)|,
    \end{align*}
    which implies $e^{-n}|P(\xi)| \leq |Q(\xi)|$, and so \eqref{eq proof: lem construction cambouis 2 bis} becomes
    \begin{align*}
        \max\left\{|Q_1(\xi)|,\dots,|Q_{n+1}(\xi)|\right\} \ll |Q(\xi)| \ll \norm{P}^{-x},
    \end{align*}
    hence \eqref{eq lem estimations de base pour Wirsing}. On the other hand, Lemma~\ref{Lem: estimation det facile} applied to the family $(Q_1,\dots,Q_{n+1})$ ensures that
    \begin{align*}
        1 \ll |Q(\xi)|\prod_{i=2 }^{n+1}\norm{Q_i} = |Q(\xi)|\cdot\norm{P}^x.
    \end{align*}
    Combined with \eqref{eq proof: lem construction cambouis 5}, this yields $|Q(\xi)| \asymp \norm{P}^{-x}$. Then~\eqref{eq proof: lem construction cambouis 3} gives $\norm{P}^{\homega} \ll \norm{P}^x$, and we conclude that $x \geq \homega_n(\xi)-\ee$ as soon as $\norm{P}$ is large enough.

    \medskip

    \noindent\textbf{Proof of assertion \ref{item:lem: estimations de base pour Wirsing:item 4}.} Assume that $P\in\PP_1(\ee)$. We now have $L_1(q_P) \leq  \log(e^{-n}\norm{P})$, and thus $q\geq q_P$. Combined with
    \begin{align*}
        e^{L(Q,q)} = e^{L_1(q)} = e^{-n}H(P) < H(P)
    \end{align*}
    this implies $|Q(\xi)| = e^{L_1(q)-q} < H(P)e^{- q_P} = |P(\xi)| \leq \norm{P}^{-\omega}$, and \eqref{eq proof: lem construction cambouis 2 bis} becomes
    \begin{align}
        \label{eq proof: lem construction cambouis 6}
        \max\left\{|Q_1(\xi)|,\dots,|Q_{n+1}(\xi)|\right\} \ll |P(\xi)| \leq \norm{P}^{-\omega}.
    \end{align}
    Together with \eqref{eq proof: lem construction cambouis 2 ter} it yields \eqref{eq lem estimations de base pour Wirsing bis}.    Finally, Lemma~\ref{Lem: estimation det facile} yields
    \begin{align*}
        1 \ll |P(\xi)|\prod_{i=2 }^{n+1}\norm{Q_i} \leq \norm{P}^{x-\omega},
    \end{align*}
    and so $x \geq \omega-\ee/2 = \omega_n(\xi)-\ee$ as soon as $\norm{P}$ is large enough.
\end{proof}

\section{Generalized resultants}
\label{Section : semi-resultants}

Let $\xi$ be a transcendental real number and $n$ be an integer $\geq 2$.
The main result of the section is the following one, which, combined with Lemma~\ref{lem: estimations de base pour Wirsing}, will allow us to construct algebraic numbers of degree at most $n$ very close to $\xi$.

\begin{Prop}
    \label{lem: premiere estimation brute pour xi-alpha}
    Let $k$ be an integer with $2\leq k \leq n+1$ and set $N= 2n-k+1 \geq n$. Let $P_1,\dots,P_k \in \bZ[X]_{\leq n}$ be linearly independent polynomials, and write $H_i = \norm{P_i}$ for $i=1,\dots,k$. Suppose that $P_1$ and $P_2$ are coprime, and that
    \[
        H_1 \leq \cdots \leq H_k \AND \max_{1\leq i\leq k}|P_i(\xi)| \leq \delta,
    \]
    for some $\delta > 0$. Then, there exists an algebraic number $\alpha$ of degree $\leq n$ and an index $m\in\{1,\dots,k\}$ such that
    \begin{align}
        \label{eq: premiere estimation brute pour xi-alpha}
        H(\alpha)\ll H_m \AND |\xi-\alpha| \ll \delta^2  H_1^{n-k+1}H_2^{n-k+2}H_3\cdots H_k H_m^{-1},
    \end{align}
    where the implicit constants depend on $n$ and $\xi$ only.
\end{Prop}

To prove this result we will use generalized resultants. Let us recall the results from \cite[§6]{poels2024pol}. We say that a function  $g : \{n,n+1,n+2 \cdots\} \rightarrow \bR$ is \textsl{concave} if
\[
    g(i)-g(i-1) \geq g(i+1)-g(i)
\]
for any $i>n$. Let $N\geq n$ be an integer and let $\cA \neq \{0\}$ be a subset of $\bR[X]_{\leq n}$ containing a non-zero element. We define
\begin{align*}
    \cB_N(\cA) &= \big\{ Q, XQ,\dots, X^{N-\deg(Q)}Q \,;\, Q\in\cA\setminus\{0\} \big\} \subset \bR[X]_{\leq N}, \\
    V_N(\cA) &= \Vect[\bR]{\cB_N(\cA)}, \\
    g_\cA(N) &= \dim V_N(\cA).
\end{align*}
We call \textsl{generalized resultant} any determinant of $N+1$ elements chosen in $\cB_N(\cA)$, for some $\cA$ as above. According to \cite[Lemma~6.3]{poels2024pol}, the function $g_\cA$ is (strictly) increasing and concave on $\{n,n+1,\dots\}$. If we assume furthermore that the $\gcd$ of the elements of $\cA$ is $1$ (in other words the ideal spanned by $\cA$ is $\bR[X]$), then
\begin{align}
    \label{eq : il existe m tel que V_m = tout}
        V_{2n-1}(\cA) = \bR[X]_{\leq 2n-1}
\end{align}
(it is a direct consequence of \cite[Proposition 6.2]{poels2024pol}).

\begin{Lem}
    \label{lem : complément sur les sous-espaces V_j}
    Let $\cA$ be a linearly independent subset of $\bR[X]_{\leq n}$ of cardinality $j$ with $2\leq j\leq n+1$. We also suppose that the $\gcd$ of the elements of $\cA$ is $1$. Then, for $k=0,\dots,n-j+1$, we have
    \begin{align*}
        \dim V_{n+k}(\cA) \geq 2k + j.
    \end{align*}
\end{Lem}

\begin{proof}
    By contradiction, suppose that there exists $k\in\{0,\dots,n-j+1\}$ such that
    \[
        g_\cA(n+k) < 2k+j.
    \]
    Since $g_\cA(n) \geq \textrm{card}(\cA) = j$, we have $k\geq 1$.

    \medskip

    \noindent\textbf{Case 1.} Suppose that $g_\cA(n+k) \geq g(n+k-1)+2$. By concavity, we have $g_\cA(i) \geq g_\cA(i-1)+2$ for $i=n+1,\dots,n+k$, and we deduce that
    \begin{align*}
        g_\cA(n+k) \geq 2k + g_\cA(n) \geq 2k+j
    \end{align*}
    which is a contradiction.

    \medskip

    \noindent\textbf{Case 2.} So  $g_\cA(n+k) \leq g(n+k-1)+1$. By concavity (and since $g_\cA$ is increasing), we have $g_\cA(i+1) = g_\cA(i)+1$ for $i=n+k,\dots,2n$. Combined with \eqref{eq : il existe m tel que V_m = tout}, we get
    \begin{align*}
        2n = g_\cA(2n-1) = g_\cA(n+k)+ 2n-1 - (n+k) & < n + k+j -1 \leq 2n
    \end{align*}
    (the last inequality coming from $k\leq n-j+1$), which is, once again, a contradiction.
\end{proof}

As a corollary, we obtain the following useful result.

\begin{Cor}
    \label{cor : complément sur les sous-espaces V_j}
    Let $k, N$ be as in Proposition~\ref{lem: premiere estimation brute pour xi-alpha}. Let $P_1,\dots,P_k \in \bZ[X]_{\leq n}$ be linearly independent polynomials such that $P_1$ and $P_2$ are coprime. Then, for each $j=2,\dots,k$ we have
    \begin{align*}
        \dim V_{N}(P_1,\dots,P_j) \geq 2(n-k+1) + j.
    \end{align*}
    In particular,
    \begin{align*}
        V_{N}(P_1,\dots,P_k) = \bR[X]_{\leq N}.
    \end{align*}
\end{Cor}

\begin{proof}[\textbf{Proof of Proposition~\ref{lem: premiere estimation brute pour xi-alpha}}]
    First, note that there exist $\lambda_1, \lambda_2\in\{0,\dots,n\}$ such that the polynomials
    \[
        Q_i = (X-\lambda_i)^{n-\deg(P_i)}P_i \qquad (i=1,2)
    \]
    are coprime and of degree exactly $n$. By Gel'fond's Lemma, they also satisfy $\norm{Q_i} \asymp \norm{P_i} =  H_i$ and $|Q_i(\xi)| \asymp |P_i(\xi)| \leq \delta$ ($i=1,2$), and the space
    \[
        F = V_N(Q_1,Q_2)
    \]
    spanned by $Q_1,XQ_1,\dots, X^{n-k+1}Q_1,Q_2,XQ_2,\dots, X^{n-k+1}Q_2$ has dimension $2(n-k+2)$. We can choose a subsequence $(Q_3,\dots,Q_k)$ of $(P_1,\dots,P_k)$ such that $Q_1,\dots, Q_k$  are linearly independent. For each $j=3,\dots,k$ there is some $i\in\{1,\dots,j\}$ such that $\norm{Q_j} = H_i \leq H_j$. According to Corollary~\ref{cor : complément sur les sous-espaces V_j}, for $j=2,\dots,k$, we have
    \[
        \dim\big(F+V_{N}(Q_3,\dots, Q_j)\big)= \dim V_{N}(Q_1,\dots, Q_j) \geq \dim F + j-2.
    \]
    For $j=k$ we obtain $V_{N}(Q_1,\dots, Q_k)=\bR[X]_{\leq N}$. By recurrence, for $j=3,\dots,k$, we can choose $R_j\in \cB_{N}(Q_3,\dots,Q_j)$ such that
    \begin{align*}
        \dim\left(F + \Vect[\bR]{R_3,\dots, R_j}\right) = \dim F + j-2,
    \end{align*}
    in particular $F + \Vect[\bR]{R_3,\dots, R_k} = \bR[X]_{\leq N}$. Note that for each $j=3,\dots,k$, there is some index $i\in\{1,\dots, j\}$ such that,
    \begin{align}
        \label{eq proof: lem: premiere estimation brute pour xi-alpha}
        \norm{R_j} = H_i \leq H_j \AND |R_j(\xi)| \ll \delta.
    \end{align}
    Moreover, the roots of $R_j$ are algebraic numbers of degree at most $n$, since they are either $0$ or a root of one of the polynomials $Q_3,\dots, Q_j \in \bZ[X]_{\leq n}$. For simplicity, write
    \begin{align*}
        (S_0,\dots, S_N) = (Q_1,XQ_1,\cdots, X^{n-k+1}Q_1,Q_2,XQ_2,\cdots, X^{n-k+1}Q_2, R_3,\dots,R_{k}).
    \end{align*}
    The first $n-k+2$ polynomials $S_i$ have norm $\asymp H_1$, while the following $n-k+2$ ones have norm~$\asymp H_2$. We control the norms of the last $k-2$ polynomials $S_i$, which are equal to the polynomials $R_j$, by using \eqref{eq proof: lem: premiere estimation brute pour xi-alpha}. The non-zero generalized resultant $\det(S_0,\dots,S_{N})$ satisfies
    \begin{align*}
        1 \leq \left|\det(S_0,\dots,S_{N})\right| = \left|\det\left( S^{[i]}_j(0) \right)_{0\leq i,j\leq N}\right|
        = \left|\det\left( S^{[i]}_j(\xi) \right)_{0\leq i,j\leq N}\right|.
    \end{align*}
    For $j=0,\dots, N$, we have
    \[
        |S^{[0]}_j(\xi)| = |S_j(\xi)| \ll \delta \AND S^{[1]}_j(\xi) = S_j'(\xi).
    \]
    For $i=2,\dots, N$ we will use the crude estimate $|S^{[i]}_j(\xi)| \ll \norm{S_j}$. Expanding the last determinant, we obtain
    \begin{align}
    \label{eq proof estimation |xi-alpha| via gros det}
        1 \leq  \left|\det\left( S^{[i]}_j(\xi) \right)_{0\leq i,j\leq N}\right| & \ll \delta H_1^{n-k+1}H_2^{n-k+2} H_3 \cdots H_k \sum_{\ell=0}^{N} \frac{|S_\ell'(\xi)|}{\norm{S_\ell}}.
    \end{align}
    Let $\ell\in\{0,\dots, N\}$ be such that $|S_\ell'(\xi)|/\norm{S_\ell}^{-1}$ is maximal, and let $\alpha$ be a root of $S_\ell$ such that $|\xi-\alpha|$ is minimal. Recall that $\alpha$ is algebraic of degree at most $n$ and that there exists $m\in\{1,\dots,k\}$ such that $\norm{S_\ell} \asymp H_m$. Then, the minimal polynomial of $\alpha$ divides $S_\ell$, and Gel'fond's lemma yields $H(\alpha) \ll \norm{S_\ell}  \ll H_m$. On the other hand
    \begin{align*}
        |\xi-\alpha||S_\ell'(\xi)| \ll |S_\ell(\xi)| \ll \delta
    \end{align*}
    (it is a classical argument, see for example \cite[Section~2]{bugeaud2007exponents}). Multiplying both sides of \eqref{eq proof estimation |xi-alpha| via gros det} by $|\xi-\alpha|$, this yields \eqref{eq: premiere estimation brute pour xi-alpha}.
\end{proof}

Proposition~\ref{lem: premiere estimation brute pour xi-alpha} has the following Corollary.

\begin{Cor}
    \label{Cor: theorem vers fonction F via semi-resultant}
    Let $k$ be an integer with $2\leq k \leq n+1$ and let $C, y > 0$. Let $P_1,\dots,P_k \in \bZ[X]_{\leq n}$ be linearly independent polynomials and write $H_i = \norm{P_i}$ for $i=1,\dots,k$. Assume that
    \begin{enumerate}
      \item $P_1$ and $P_2$ are coprime and $H_2 \geq 2$;
      \item $H_1\leq \cdots \leq H_k$; 
      \item $|P_i(\xi)| \leq C H_2^{-y}$ for $i=1,\dots,k$.
    \end{enumerate}
    For $i=1,\dots,k$ write $H_i = H_2^{a_i}$, and suppose furthermore that
    \[
        A_k := 2y-2(n+1-k)-a_2-\cdots - a_k \geq 0.
    \]
    Then, there exists an algebraic number $\alpha$ of degree $\leq n$ and a constant $c$ which depends on $n$, $\xi$ only, such that
    \begin{align}
        \label{eq thm : exposant A a la base de fonction F}
        |\xi-\alpha| \ll C^2\min\left\{\big(cH(\alpha)\big)^{-A_k/a_k-1}, H_2^{-A_k-1}\right\}.
    \end{align}
    The implicit constant depends on $n$ and $\xi$ only.
\end{Cor}

\begin{Rem}
    Since $A_k\geq 0$, equation~\eqref{eq thm : exposant A a la base de fonction F} implies that $|\xi-\alpha| \ll 1/H_2$ tends to $0$ as $H_2$ tends to infinity. Consequently $H(\alpha)$ tends to infinity as $H_2$ tends to infinity.
\end{Rem}

\begin{proof}
    Set $\delta = CH_2^{-y}$. By Proposition~\ref{lem: premiere estimation brute pour xi-alpha}, there exists an algebraic number $\alpha$ of degree at most $n$ and $m\in\{2,\dots,k\}$ such that
    \begin{align}
        \label{eq proof: Cor: theorem vers fonction F via semi-resultant}
        cH(\alpha)\leq H_m \AND |\xi-\alpha| \ll \delta^2  H_2^{2n-2k+3}H_3\cdots H_kH_m^{-1} = C^2H_2^{-A_k-a_m},
    \end{align}
    where $c>0$ depends on $\xi$ and $n$ only. Since $a_m\geq 1$, we have $|\xi-\alpha| \ll C^2H_2^{-A_k-1}$. Furthermore, using $a_k\geq a_m$ and $A_k\geq 0$, Estimates~\eqref{eq proof: Cor: theorem vers fonction F via semi-resultant} also yield
    \begin{align*}
        |\xi-\alpha| \ll  C^2H_m^{-A_k/a_m-1} \leq C^2H_m^{-A_k/a_k-1} \leq C^2(cH(\alpha))^{-A_k/a_k-1}.
    \end{align*}
\end{proof}

\section{A step toward Wirsing's conjecture}
\label{section: introduction de la fonction F}

Let $\xi$ be a transcendental real number and $n \geq 2$ be an integer. In this section, we merge the main results of the preceding two sections to provide a lower bound for $\omega_n^*(\xi)$. This uses the following notation. Given $x\geq n$ we define
\[
    \cA(x) = \Big\{\balpha=(a_2,\dots,a_{n+1})\in\bR^n \,;\, 1 = a_2 \leq \cdots \leq a_{n+1} \AND a_2+ \cdots + a_{n+1} = x \Big\}.
\]
For each $\balpha = (a_2,\dots,a_{n+1})\in \cA(x)$ and each integer $k$ with $2 \leq k \leq n+1$, we set
\begin{align*}
    A_k(x,\balpha) & = 2x - 2(n-k+1)- \sum_{i=2}^k a_i= 2(x-n) + \sum_{i=2}^k (2-a_i),
\end{align*}
and
\begin{align*}
    F(x,\balpha) = \max_{2\leq k \leq n+1} \frac{A_k(x,\balpha)}{a_k}.
\end{align*}
Since $\balpha \mapsto F(x,\balpha)$ is continuous on the compact set $\cA(x)$, we may also define
\[
    F(x) = \min_{\ba\in\cA(x)} F(x,\balpha).
\]
Note that the condition $a_2+\cdots + a_{n+1} = x$ in the definition of $\cA(x)$ is equivalent to $x = A_{n+1}(x,\ba)$. Furthermore, for each $\balpha\in\cA(x)$, we have
\begin{align}
    \label{eq: minoration A_k par 1}
    A_k(x,\balpha) \geq 1 \qquad \textrm{for } k=2,\dots,n+1,
\end{align}
since $2(n-k+1)+ a_2+\cdots + a_k \leq (n-k+1) + a_2 + \cdots + a_{n+1} \leq 2x-1$.

\begin{Thm}
    \label{Thm : omega_n^* >= F_0}
    We have
    \[
        \omega_n^*(\xi) \geq \minF:=\inf_{x\geq n} F(x).
    \]
\end{Thm}

\begin{proof}
    If $\omega_n(\xi) = \infty$, then $\omega_n^*(\xi) = \infty$ (see Remark~\ref{Rem: section notation, exposants non extremal}) and we are done. We may therefore suppose that $\omega_n(\xi) < \infty$. Fix a small $\ee \in (0,1/2)$ and let $P$ be an element of the infinite set $\PP(\ee)$ defined as in Section~\ref{section : construction points à considérer}. According to Lemma~\ref{lem: estimations de base pour Wirsing}, if $\norm{P}$ is large enough, then there exist linearly independent polynomials $Q_1,\dots,Q_{n+1}\in\bZ[X]_{\leq n}$ and $x\geq n$ such that, writing $H_i = \norm{Q_i}$ for $i=1,\dots,n+1$, we have
    \begin{enumerate}
        \item $Q_1$ et $Q_2$ are coprime, with $Q_2 = P$;
        \item \label{enum preuve thm minoration thm gen: enum 2} $H_1 \leq  \dots \leq H_{n+1}$ and $H_2\cdots H_{n+1} = H_2^x$;
        \item $|Q_1(\xi)|,\dots,|Q_{n+1}(\xi)| \ll H_2^{-x+\ee}$.
    \end{enumerate}
    For $i=2,\dots,n+1$, define $a_i \geq 1$ by $H_i = H_2^{a_i}$. Condition~\ref{enum preuve thm minoration thm gen: enum 2} means that the point $\balpha = (a_2,\dots,a_{n+1})$ belongs to
    $\cA(x)$.
    Set $y=x-\ee$. By \eqref{eq: minoration A_k par 1}, for each $k\in\{2,\dots,n+1\}$,  we have
    \begin{align*}
        2y-2(n+1-k)-a_2-\cdots - a_k = A_k(x,\balpha) - 2\ee \geq 0.
    \end{align*}
    By Corollary~\ref{Cor: theorem vers fonction F via semi-resultant} applied successively with $k=2,\dots,n+1$,
    there exists an algebraic number $\alpha$ of degree at most $n$, such that
    \begin{align*}
        |\xi-\alpha| \ll H(\alpha)^{-F(x,\balpha)-1+2\ee} \leq H(\alpha)^{-F(x)-1+2\ee} \leq  H(\alpha)^{-\minF-1+2\ee}.
    \end{align*}
    Recall that $H(\alpha)$ tends to infinity with $\norm{P}$, because we have $|\xi-\alpha| \ll H_2^{-1} = \norm{P}^{-1}$ in view of the last estimate of Corollary~\ref{Cor: theorem vers fonction F via semi-resultant}. Since Since $\PP(\ee)$ is infinite, we deduce that
    \begin{align*}
        \omega_n^*(\xi) \geq \minF-2\ee,
    \end{align*}
    and we get the result by letting $\ee$ tend to $0$.
\end{proof}

\section{A minimization problem}
\label{Section: points realisant le min}

Let the notation be as in Section~\ref{section: introduction de la fonction F}. Theorem~\ref{Thm : omega_n^* >= F_0} calls for a lower bound estimate for $\minF$. We first prove that there exists a point $(x,\ba)\in\bR^{n+1}$ with $x\geq n$ and $\ba\in\cA(x)$ satisfying $F(x, \ba) = \minF$. Then, we give a complete description of $\ba$ as a function of $\minF$ and $x$ and some integer $\ell$ with $2\leq \ell \leq n$. In the final Section~\ref{section: minration finale}, we use these properties to give an explicit lower bound for $\minF$ and deduce Theorem~\ref{Thm : main}. Our approach is inspired by the remarkable strategy described by de La Vallée-Poussin in \cite[Chapter VI]{DeLaValleePoussinApprox1919} to construct polynomials of best approximation to a continuous real valued function on a closed interval on $\bR$.

\begin{Thm}
    \label{Thm : description du min}
    There exists a point $(x,\ba)\in \bR^{n+1}$, with $\ba = (a_2,\dots,a_{n+1})$, such that
    \begin{align}
        \label{eq:Thm : description du min:eq1}
       x\geq n, \quad \ba\in\cA(x) \AND \minF = F(x,\ba).
    \end{align}
    Any such point has the following properties.
    \begin{enumerate}
        \item \label{item:Thm : description du min:item 1} There exists $\ell \in \{2,\dots, n\}$ such that $\minF = 2(x-n)+\ell-1$ and
            \begin{align*}
                x = (2-\theta)\minF, \quad \textrm{where } \theta = \left(\frac{\minF}{\minF+1}\right)^{n+1-\ell}.
            \end{align*}
        \item \label{item:Thm : description du min:item 2} The point $\ba = (a_2,\dots,a_{n+1})$ is given by $a_2 = \cdots = a_\ell = 1$, and
        \begin{align*}
            a_k = 2 - \left(\frac{\minF}{\minF+1} \right)^{k-\ell}, \qquad \textrm{for } k=\ell,\dots,n+1.
        \end{align*}
        \item \label{item:Thm : description du min:item 3} We have
        \begin{align*}
            \minF =  \frac{A_{\ell+1}(x,\ba)}{a_{\ell+1}} = \cdots =  \frac{A_{n+1}(x,\ba)}{a_{n+1}}.
        \end{align*}
    \end{enumerate}
\end{Thm}

Theorem~\ref{Thm : description du min} implies that there are at most $n-1$ points satisfying \eqref{eq:Thm : description du min:eq1} (for such a point is entirely determined by the integer $\ell$). Note that the first part of \ref{item:Thm : description du min:item 1} combined with $a_\ell = 1$ ensures that the formula in \ref{item:Thm : description du min:item 3} is also valid for the index $\ell$. In order to prove the above theorem, we first prove that the infimum $\minF$ is actually a minimum.

\begin{Lem}
    \label{Lem: ens des points realisant le min compact non vide}
    We have $F_n < n$, and the set $\cM_n$ of points $(x,\ba)\in\bR^{n+1}$ satisfying \eqref{eq:Thm : description du min:eq1} is non-empty.
    Furthermore, any $(x,\ba)\in \cM_n$ has $n < x < (3n-1)/2$.
\end{Lem}

\begin{proof}
    For a fixed $\ee\in[0,1/2)$, the point $\ba=(1,\dots,1,1+\ee)\in\bR^n$ belongs to $\cA(x)$ with $x=n+\ee$. It follows from the definition that $A_k(x,\ba) = 2\ee+k-1 < n$ for $k=2,\dots,n$ and
    \begin{align*}	
    	 A_{n+1}(x,\ba) = \frac{n+\ee}{1+\ee} \leq n,
    \end{align*}
    with equality if and only if $\ee=0$. Taking $0 < \ee < 1/2$, we deduce that $\minF \leq F(x,\ba) < n$. Note that for $x=n$, the set
    $\cA(n)$ reduces to $\{(1,\dots,1)\}$ and $F(n) = n$. On the other hand, for any $x\geq n$, each $\ba = (a_2,\dots,a_{n+1})\in\cA(x)$ has $a_2=1$, thus
    \begin{align*}
        F(x,\balpha) \geq \frac{A_2(x,\balpha)}{a_2} = 2(x-n)+1.
    \end{align*}
    If follows that $F(x) \geq 2(x-n)+1$. Consequently, if $x\geq (3n-1)/2$, then $F(x) \geq n > \minF$. Consider the compact subset $\cK_n$ of $\bR^{n+1}$ given by
    \begin{align*}
        \cK_n =\left\{ (x,\balpha)\in\bR^{n+1} \;|\; x\in \left[ n,\frac{3n-1}{2} \right] \AND \balpha\in\cA(x) \right\}.
    \end{align*}
    By the above, we have $\minF = \inf_{(x,\balpha)\in \cK_n} F(x,\balpha)$. Since the function $F$ is continuous on the compact set $\cK_n$,    this infimum is actually a minimum. Furthermore, since $\minF < n$, any point $(x,\ba)\in \cK_n$ realizing this minimum satisfies $n < x < (3n-1)/2$.
\end{proof}

\begin{Lem}
    \label{Lem : premiere prop du point realizant le min}
    Let $\cM_n$ be as in Lemma~\ref{Lem: ens des points realisant le min compact non vide}, let $(x,\ba)\in\cM_n$ and
    write $\ba = (a_2,\dots,a_{n+1})$. There exists an integer $\ell\in\{2,\dots,n\}$ such that
    \begin{enumerate}
        \item \label{item: Lem : premiere prop du point realizant le min: item 1} $1 = a_2 = \cdots = a_\ell < a_{\ell+1} < \cdots < a_{n+1} < 2$;
        \item \label{item: Lem : premiere prop du point realizant le min: item 2} $A_{\ell+1}(x,\ba)/a_{\ell+1} = \cdots = A_{n+1}(x,\ba)/a_{n+1} = \minF$;
        \item \label{item: Lem : premiere prop du point realizant le min: item 3} $2(x-n)+\ell-1 \leq \minF < 2(x-n)+\ell$.
    \end{enumerate}
\end{Lem}

\begin{proof}
    \noindent\textbf{Step $1$.} Suppose that $a_j < a_{j+1}$ for an integer $j$ with $2\leq j \leq n$. We claim that
    \begin{align}
        \label{eq proof: eq 0 lemme inter tete du point realisant le min}
        \frac{A_{j+1}(x,\ba)}{a_{j+1}} = \minF.
    \end{align}
    Indeed, for each $\ee \in (0,a_{j+1}-a_j]$, the point
    \begin{align*}
        \bb = (b_2,\dots, b_{n+1}) = (a_2,\dots,a_j,a_{j+1}-\ee,a_{j+2},\dots,a_{n+1})
    \end{align*}
    belongs to $\cA(y)$, where $y = x-\ee$. Since $b_2 = 1$, we have $y = b_2+\cdots + b_{n+1} \geq n$. By definition of the functions $A_k$
    we have
    \begin{align*}
        A_k(y,\bb) = \left\{\begin{array}{ll}
          A_k(x,\ba)-2\ee & \textrm{for } k=2,\dots,j, \\
          A_k(x,\ba)-\ee & \textrm{for } k=j+1,\dots,n+1.
        \end{array} \right.
    \end{align*}
    So, for each $k\neq j+1$, we find
    \begin{align*}
        \frac{A_k(y,\bb)}{b_k} < \frac{A_k(x,\ba)}{a_k} \leq F(x,\ba) = \minF.
    \end{align*}
    However, by minimality of $F_n$, we have $F(y,\bb) \geq \minF$, thus
    \begin{align*}
        F(y,\bb) = \frac{A_{j+1}(y,\bb)}{b_{j+1}} = \frac{A_{j+1}(x,\ba)-\ee}{a_{j+1}-\ee} \geq \minF.
    \end{align*}
    Letting $\ee$ tend to $0$, we obtain $\minF \leq A_{j+1}(x,\ba)/a_{j+1} \leq F(x,\ba)$, hence our claim.

    \medskip

    \noindent\textbf{Step $2$.} Suppose that $a_{n+1} \geq 2$. Since $a_2 = 1 < 2$, there exists an integer $j$ with $2\leq j \leq n$ such that $a_j < 2 \leq a_{j+1}$.  Using Step~$1$, we get
    \begin{align*}
        \minF = F(x,\ba) \geq \frac{A_j(x,\ba)}{a_j} = \frac{A_{j+1}(x,\ba)+a_{j+1}-2}{a_j} > \frac{A_{j+1}(x,\ba)}{a_{j+1}} = \minF,
    \end{align*}
    a contradiction. Hence $1 = a_2 \leq a_3 \leq \cdots \leq a_{n+1} < 2$.

    \medskip

    \noindent\textbf{Step $3$.} Let $\ell$ be the largest integer in $\{2,\dots,n+1\}$ such that $a_\ell = 1$. As  $1 = a_2 \leq \cdots \leq a_{n+1}$, we have $a_2 = \cdots = a_\ell = 1$. If $\ell = n+1$, then $x=n$, which contradicts Lemma~\ref{Lem: ens des points realisant le min compact non vide}, so $\ell\leq n$. By contradiction, suppose that assertion~\ref{item: Lem : premiere prop du point realizant le min: item 1} is false. Then $\ell < n$ and by Step~$2$ there exists an integer $j$ with $\ell \leq j \leq n-1$ such that $a_j < a_{j+1} = a_{j+2} < 2$. Using Step~$1$, we obtain
        \begin{align*}
        \minF = F(x,\ba) \geq \frac{A_{j+2}(x,\ba)}{a_{j+2}} = \frac{A_{j+1}(x,\ba)+2-a_{j+2}}{a_{j+1}} > \frac{A_{j+1}(x,\ba)}{a_{j+1}} = \minF,
    \end{align*}
    a contradiction. Thus \ref{item: Lem : premiere prop du point realizant le min: item 1} holds, and by Step~$1$, it yields \ref{item: Lem : premiere prop du point realizant le min: item 2}. Finally,
    assertion~\ref{item: Lem : premiere prop du point realizant le min: item 3} follows from $a_{\ell+1}>1$ and
    \begin{align*}
        2(x-n)+\ell-1 = \frac{A_\ell(x,\ba)}{a_\ell} \leq \minF = \frac{A_{\ell+1}(x,\ba)}{a_{\ell+1}} < A_{\ell+1}(x,\ba) < 2(x-n)+\ell.
    \end{align*}
\end{proof}

\begin{Lem}
    \label{Lem : second prop du point type point-min}
    Let $x\in[n,(3n-1)/2]$, let $\ell\in\{2,\dots,n\}$, let $\ba= (a_2,\dots,a_{n+1})\in\bR^n$ with $a_2 = \cdots = a_\ell = 1$ and let $y, F\in\bR$ with
    \[
        y = 2(x-n)+\ell-1 \leq F < y+1.
    \]
    The following assertions are equivalent
    \begin{enumerate}
      \item \label{item:Lem: second prop du point type point-min:item 1} For $k=\ell+1,\dots,n+1$, we have
        \begin{align}
            \label{eq Lem : second prop du point type point-min : eq 1}
            \frac{A_{k}(x,\ba)}{a_{k}} = F.
        \end{align}
      \item \label{item:Lem: second prop du point type point-min:item 3} For $k=\ell+1,\dots,n+1$, we have
        \begin{align}
            \label{eq Lem : second prop du point type point-min : eq 3}
            a_k = 2 - \frac{2F-y}{F+1}\left(\frac{F}{F+1}\right)^{k-\ell-1}.
        \end{align}
    \end{enumerate}
    If they hold, then $1 < a_{\ell+1} < \cdots < a_{n+1}<2$.
\end{Lem}

\begin{proof}
    \ref{item:Lem: second prop du point type point-min:item 1} $\Leftrightarrow$ \ref{item:Lem: second prop du point type point-min:item 3}. As $A_{\ell+1} = y+2-a_{\ell+1}$, we first observe that \eqref{eq Lem : second prop du point type point-min : eq 1}
    holds for $k=\ell+1$ if and only if $a_{\ell+1} = (y+2)/(F+1)$. Suppose that \eqref{eq Lem : second prop du point type point-min : eq 1} holds for an index $k$ with $\ell+1 \leq k \leq n$. Then, since $A_{k+1}(x,\ba) = A_{k}(x,\ba)+ 2 - a_{k+1}$, the equality \eqref{eq Lem : second prop du point type point-min : eq 1} holds for $k+1$ if and only if
    \begin{align}
        \label{eq Lem : second prop du point type point-min : eq 2}
        a_{k+1} = \frac{F}{F+1} a_k+\frac{2}{F+1}.
    \end{align}
    By the above remark, \ref{item:Lem: second prop du point type point-min:item 1} holds if and only if $a_{\ell+1} = (y+2)/(F+1)$ and \eqref{eq Lem : second prop du point type point-min : eq 2} is satisfied for $k=\ell+1,\dots,n$. This is precisely the arithmetico-geometric sequence of \ref{item:Lem: second prop du point type point-min:item 3}.

    \medskip

    Finally, the hypothesis $y\leq F < y+1$ implies that
	\begin{align*}
		\frac{F}{F+1} \leq \frac{2F-y}{F+1} < 1,
	\end{align*}
    so that if \eqref{eq Lem : second prop du point type point-min : eq 3} holds, then
    $a_2 = \cdots = a_\ell < a_{\ell+1} < \cdots < a_{n+1}$.
\end{proof}

\begin{Lem}
    \label{Lem : second prop du point type point-min - bis}
    With the notation and hypotheses of Lemma~\ref{Lem : second prop du point type point-min} set
    \begin{align*}
    	\theta = \left(\frac{F}{F+1}	\right)^{n+1-\ell},
    \end{align*}
    and suppose that $\ba\in\bR^n$ satisfies the two equivalent conditions \ref{item:Lem: second prop du point type point-min:item 1} and \ref{item:Lem: second prop du point type point-min:item 3}
    of Lemma~\ref{Lem : second prop du point type point-min}, as well as $1=a_2= \cdots = a_\ell$. Then, $(x,\ba)\in\cA(x)$ if and only if
    \begin{align}
    \label{eq Lem : second prop du point type point-min : eq 4}
    	(1-2\theta)x = 2F(1-\theta)-\theta(2n-\ell+1).
    \end{align}
    Moreover, if \eqref{eq Lem : second prop du point type point-min : eq 4} holds, then $\theta\neq 1/2$.
\end{Lem}

\begin{proof}
	Since $y = 2(x-n)+\ell-1$, Equation \eqref{eq Lem : second prop du point type point-min : eq 4} is equivalent to $x=2F - (2F-y)\theta$.
    Recall that $a_2+\cdots+a_{n+1} = x$ if and only if $x = A_{n+1}(\ba,x)$. Since the coordinates of $\ba$ are increasing, the point $(x,\ba)$ belongs to $\cA(x)$ if, and only if,
    \begin{align*}
        x = A_{n+1}(x,\ba) \mathop{=}_{\eqref{eq Lem : second prop du point type point-min : eq 1}} a_{n+1}F
        \mathop{=}_{\eqref{eq Lem : second prop du point type point-min : eq 3}} 2F - \frac{F(2F-y)}{F+1}\left(\frac{F}{F+1}\right)^{n-\ell}
        = 2F - (2F-y)\theta.
    \end{align*}
    It remains to prove the last part of the Lemma. By contradiction, suppose now that \eqref{eq Lem : second prop du point type point-min : eq 4} holds with $\theta=1/2$. Then, we obtain
    \begin{align*}
        F = \frac{2n-\ell+1}{2} \in \bQ \AND
        \frac{1}{2} = \theta = \left(\frac{F}{F+1}\right)^{n+1-\ell}.
    \end{align*}
    This implies that the exponent $n+1-\ell$ is equal to $1$, thus $\ell = n$ and $F = (n+1)/2 = 1$, which is impossible since $n\geq 2$.
\end{proof}

\begin{proof}[Proof of Theorem~\ref{Thm : description du min}]
    By Lemma~\ref{Lem: ens des points realisant le min compact non vide} there exists $(\minx,\minba) = (x,a_2,\dots,a_{n+1})\in\bR^{n+1}$ satisfying \eqref{eq:Thm : description du min:eq1}, and any such point has $n < \minx < (3n-1)/2$. Fix such a point. Then Lemma~\ref{Lem : premiere prop du point realizant le min} provides an integer $\ell\in\{2,\dots,n\}$ for which $a_2 = \cdots = a_\ell = 1$,
    \begin{align}
    	\label{eq proof:thm description du min:eq1}
        \miny:= 2(\minx-n)+\ell-1 \leq \minF < \miny+1,
    \end{align}
    and assertion~\ref{item:Thm : description du min:item 3} of Theorem~\ref{Thm : description du min} holds. It only remains to prove that $\minF=\miny$ and that $x=(2-\theta)\minF$, for then Lemma~\ref{Lem : second prop du point type point-min} implies assertion~\ref{item:Thm : description du min:item 2} of the theorem. According to Lemma~\ref{Lem : second prop du point type point-min - bis}, we have
    \begin{align*}
        (1-2\mint)\minx = 2\minF(1-\mint)-\mint(2n-\ell+1),\quad \textrm{where } \mint = \left(\frac{\minF}{\minF+1}\right)^{n+1-\ell} \neq \frac{1}{2}.
    \end{align*}
    Fix $\ee \in [0,1)$ and set $F' = \minF-\ee$. If $\ee$ is small enough, then
    \begin{align*}
         \theta' := \left(\frac{F'}{F'+1}\right)^{n+1-\ell} \neq \frac{1}{2},
    \end{align*}
    and there exists $x' = x'(\ee)\in\bR$ such that $(x',\theta',F')$ satisfy \eqref{eq Lem : second prop du point type point-min : eq 4}. By contradiction, suppose that $\miny < \minF < \miny+1$. We note that for $\ee=0$, we have $(x',F')=(\minx,\minF)$. So, if $\ee$ is small enough, we also have $n < x' < (3n-1)/2$ and $y' < F' < y'+1$, where
    \begin{align*}
    	y' = y'(\ee) = 2(x'-n)+\ell-1.
    \end{align*}
    Set $a_2' = \cdots = a_\ell' = 1$ and define $a_k'$ by \eqref{eq Lem : second prop du point type point-min : eq 3} (with $F=F'$) for $k=\ell+1,\dots,n+1$. We denote by $\ba'$ the point $(a_2',\dots,a_{n+1}')$. Then $x'$, $\ba'$, $\ell$, $y'$ and $F'$ satisfy the hypotheses of Lemmas~\ref{Lem : second prop du point type point-min} and~\ref{Lem : second prop du point type point-min - bis}.
    According to Lemma~\ref{Lem : second prop du point type point-min - bis}, and since $x'$, $F'$ satisfy \eqref{eq Lem : second prop du point type point-min : eq 4}, we have $\ba'\in\cA(x')$. Moreover, $A_k(x',\ba')/a_k' = 2(x'-n)+k-1 \leq y'$ for $k=2,\dots,\ell$, and our choice of $\ba'$ yields
    \begin{align*}
    	\frac{A_k(x',\ba')}{a_k'} = F' \geq y' \qquad \textrm{for } k=\ell+1,\dots,n+1.
    \end{align*}
    Thus $F(x',\ba') = F' < \minF$, a contradiction. It follows that $\minF = \miny$, as expected, hence assertion~\ref{item:Thm : description du min:item 2} of Theorem~\ref{Thm : description du min} holds. In particular, the last coordinate of $\minba$ multiplied by $\minF$ is equal to $(2-\mint)\minF$ by assertion~\ref{item:Thm : description du min:item 2}, and is also equal to $A_{n+1}(\minx,\minba) = \minx$ by assertion~\ref{item:Thm : description du min:item 3}. Hence the identity $(2-\mint)\minF = \minx$.
\end{proof}

\section{Proof of the main result}
\label{section: minration finale}

Let $n$ be an integer $\geq 2$. We keep the notation of Section~\ref{section: introduction de la fonction F} for the function $F$ and its minimum $\minF$. We now have all the tools we need to give an explicit lower bound for $\minF$. Together with Theorem~\ref{Thm : omega_n^* >= F_0}, the next estimate implies Theorem~\ref{Thm : main}.

\begin{Thm}
    We have $\minF \geq n/(2-\log 2)$.
\end{Thm}

\begin{proof}
    Fix $(x,\ba)\in\bR^{n+1}$ satisfying the condition \eqref{eq:Thm : description du min:eq1} of Theorem~\ref{Thm : description du min}, and let $\ell\in\{2,\dots,n\}$ such that
    \begin{align}
        \label{eq proof main: eq1}
        \minF = 2(x-n)+\ell-1.
    \end{align}
    Set $\theta = \big(\minF/(\minF+1)\big)^{n+1-\ell}$. The formula $x= (2-\theta)\minF$ combined with \eqref{eq proof main: eq1}
    leads to
    \begin{align}
        \label{eq proof main: eq2}
        (3-2\theta)\minF = 2n+1-\ell.
    \end{align}
    Since  $t\log(1+1/t)\leq 1$ for each $t>0$, we find
    \begin{align*}
        \minF\log \theta = -(n+1-\ell)\minF\log\left(1+\frac{1}{\minF}\right) \geq -(n+1-\ell).
    \end{align*}
    Together with \eqref{eq proof main: eq2}, this yields $(3-2\theta+\log \theta)\minF \geq n$. Finally, the function $t\mapsto 3-2t+\log t$ has a global maximum on $(0,\infty)$ at $t=1/2$, which is equal to $2-\log 2$, hence $(2-\log 2)\minF \geq n$.
\end{proof}

\Ack


\bibliographystyle{abbrv}



\footnotesize {

}

\Addresses

\end{document}